\documentclass[11pt, letterpaper]{article}
\usepackage{amsmath, amsfonts, amssymb, amsthm}
\usepackage{mathtools}
\usepackage[margin=1in]{geometry}
\usepackage{setspace}
\usepackage{authblk}
\usepackage{cite}
\usepackage[utf8]{inputenc}
\usepackage[T1]{fontenc}
\usepackage{lmodern}
\usepackage{lineno}
\usepackage{enumitem}
\usepackage[colorlinks=true, allcolors=blue]{hyperref}
\usepackage{xcolor}
\usepackage{authblk}
\newtheorem{theorem}{Theorem}[section]

\newtheorem{lemma}[theorem]{Lemma}

\newtheorem{proposition}[theorem]{Proposition}

\numberwithin{equation}{section}
\newtheorem{maintheorem}{Theorem}

\usepackage{authblk}

\author[1]{Laith Hawawsheh}
\author[2]{Ahmad Al-Salman\thanks{Corresponding author. Email: \texttt{alsalman@yu.edu.jo}}}
\author[3]{Yibiao Pan}

\affil[1]{\small School of Computing, German Jordanian University, Amman, Jordan\\
Email: \texttt{Laith.hawawsheh@gju.edu.jo}}

\affil[2]{\small Department of Mathematics, Yarmouk University, Irbid, Jordan\\
Email: \texttt{alsalman@yu.edu.jo}}

\affil[3]{\small Department of Mathematics, University of Pittsburgh, USA\\
Email: \texttt{yibiao@pitt.edu}}

\allowdisplaybreaks
\begin{document}

\title{On Block Spaces on the Unit Sphere}
\date{}
\maketitle
\begin{abstract}
In this paper, we establish several properties of the block spaces $B_q^{0,v}$. Our primary contribution resolves an open question regarding the relationship between $B_q^{0,v}$ and the Zygmund-type space $L(\log L)^{1+v}$ by demonstrating that $B_q^{0,v}$ is a proper subspace of $L(\log L)^{1+v}$. Furthermore, we prove that the union of all block spaces $B_q^{0,v}$ for $1 < q \le \infty$ is strictly contained within $L(\log L)^{1+v}$.
\end{abstract}

\section{Introduction}

For $n\geq 2$, let $\mathbb{S}^{n-1}$ be the unit sphere in the $n$
-dimensional Euclidean space $\mathbb{R}^{n}$, and let $d\sigma $ be the
induced Lebesgue measure on $\mathbb{S}^{n-1}$. In their study\ of the
convergence of the Fourier series, M. H. Taibleson and G. Weiss \cite{TW}
introduced the method of block decomposition for functions. Subsequently,
several applications of the block decomposition have been investigated. For
background information on block spaces, their applications, and their
theory, we refer the reader to \cite{ALQALS-ASIAN}--\cite{TW}, and
references therein. For $1<q<\infty $, the block spaces $B_{q}^{0,\nu}$ on $
\mathbb{S}^{n-1}$ are defined by 
\begin{equation*}
B_{q}^{0,\nu}(\mathbb{S}^{n-1})=\left\{ \Omega \in L^{1}(\mathbb{S}
^{n-1}):\Omega (\theta )=\sum_{j=1}^{\infty }\mu _{j}a_{j}(\theta
),\,A_{q}^{0,\nu}(\{\mu _{j}\})<\infty \right\} ,
\end{equation*}
where each $\mu _{j}$ is a complex number, each $a_{j}$ is a $q$-block
supported in a cap $M_{j}$ on $\mathbb{S}^{n-1}$ with $|M_{j}|=\sigma
(M_{j}) $, 
\begin{equation*}
\left\Vert a_{j}\right\Vert _{q}\leq |M_{j}|^{-\frac{1}{q^{\prime }}}
\end{equation*}
and 
\begin{equation*}
A_{q}^{0,\nu}(\{\mu _{j}\})=\sum_{j=1}^{\infty }|\mu _{j}|\left\{ 1+\left(
\log ^{+}\frac{1}{|M_{j}|}\right) ^{1+\nu }\right\}\coloneqq\sum_{j=1}^{\infty }|\mu _{j}|\,b_{j}.
\end{equation*}
Here $\frac{1}{q}+\frac{1}{q^{\prime }}=1$.

It is shown in \cite{LTW} and \cite{KS} that for $1<q\leq \infty $, we have 
\begin{equation*}
B_{q_{2}}^{0,0}\subseteq B_{q_{1}}^{0,0}\mbox{
\quad (}1<q_{1}<q_{2})
\end{equation*}
and 
\begin{equation*}
L^{q}\left( \mathbb{S}^{n-1}\right) \subseteq B_{q}^{0,0}\left( \mathbb{S}
^{n-1}\right) .
\end{equation*}

A natural question that often arises in the context of singular integral
operators, and in particular Marcinkiewicz integral operators, concerns the
relationship between the block space $B_{q}^{0,0}\left( \mathbb{S}
^{n-1}\right) $ and the Hardy space $H^{1}(\mathbb{S}^{n-1})$ in the sense
of Coifman and Weiss \cite{coifman}. On the other hand, Calder\'{o}n and
Zygmund showed in the their fundamental papers \cite{CZ1,CZ2} that the
classical Calder\'{o}n and Zygmund singular operator 
\begin{equation*}
T_{\Omega }f(x)=\mathrm{p.v.}\,\int\limits_{\mathbb{R}^{n}}f(x-y)\,\frac{
\Omega (y)}{|y|^{\,n}}dy.
\end{equation*}
is bounded on $L^{p}$ for $1<p<\infty $ under the condition that the kernel $
\Omega $ is in $L\log ^{+}L(\mathbb{S}^{n-1})$ , i.e. 
\begin{equation}
\int_{\mathbb{S}^{n-1}}\Omega (y') \log^+\left\vert \Omega (y^{\prime
})\right\vert d\sigma (y^{\prime })<\infty .  \label{1.3}
\end{equation}
Here, $\Omega $ is a homogeneous function of degree zero on $\mathbb{R}^{n}$
satisfying 
\begin{equation}
\int\limits_{\mathbb{S}^{n-1}}\Omega (x)d\sigma (x)=0.  \label{1.8}
\end{equation}
Moreover, it is shown in \cite{CZ1,CZ2} that the condition \ $\Omega\in
L\log ^{+}L(\mathbb{S}^{n-1})$ is optimal in the sense that the
corresponding operator might not be bounded if the kernel $\Omega $ is
assumed to be in $\Omega \in L\left( \log ^{+}L\right) ^{1-\varepsilon }(
\mathbb{S}^{n-1})\backslash L\log ^{+}L(\mathbb{S}^{n-1})$ for some $
\varepsilon >0$. \ It is known that the class $L\log ^{+}L(\mathbb{S}^{n-1})$
is a proper subspace of \ $H^{1}(\mathbb{S}^{n-1})$.

In \cite{BLOCK RELATION}, Ye and Zhu proved the following result:

\begin{maintheorem}[\cite{BLOCK RELATION}]\label{thm:A}
	For $1<q<\infty$ and $\nu>-1$, the space
	$B_{q}^{0,\nu}(\mathbb{S}^{n-1})$ is a proper subspace of the space
	$H^{1}(\mathbb{S}^{n-1})+L(\log^{+}L)^{1+\nu}(\mathbb{S}^{n-1})$.
\end{maintheorem}

As a particular special case of Theorem \ref{thm:A}, we obtain that $B_{q}^{0,0}(
\mathbb{S}^{n-1})$ is a proper subspace of $H^{1}(\mathbb{S}^{n-1})$.
However, the relationship between the block spaces $B_{q}^{0,\nu }(\mathbb{S}
^{n-1})$ \ and the space $L\left( \log ^{+}L\right) ^{1+\nu }(\mathbb{S}
^{n-1})$ was left open in \cite{BLOCK RELATION} and, to the best of our
knowledge, has not been addressed in the subsequent literature. In this
note, we prove the following result:

\begin{maintheorem}\label{thm:B}
Let $1<q<\infty $ and $\nu>-1$. Then $B_{q}^{0,\nu }(
\mathbb{S}^{n-1})\subset L\left( \log ^{+}L\right) ^{1+\nu }(\mathbb{S}^{n-1}). $
\end{maintheorem}

Theorem \ref{thm:B} gives a precise characterization of the position of $B_{q}^{0,0}(
\mathbb{S}^{n-1})$ within the scale of classical function spaces on $\mathbb{
S}^{n-1}$. In fact, we have the following chain of inclusions:
\begin{equation}
\mathcal{C}^{1}(\mathbb{S}^{n-1})\subsetneqq Lip_{\alpha }(\mathbb{S}
^{n-1})\subsetneqq L^{q}(\mathbb{S}^{n-1})\subsetneqq B_{q}^{0,0}(\mathbb{S}
^{n-1})\subsetneqq L(\log ^{+}L)(\mathbb{S}^{n-1})\subsetneqq H^{1}(
\mathbb{S}^{n-1})\subsetneqq L^{1}(\mathbb{S}^{n-1})  \label{1.16}
\end{equation}
for $0<\alpha <1$ and $q>1$. A further and more interesting consequence of
Theorem B is the following inclusion: 
\begin{equation*}
B_{q}^{0,-\frac{1}{2}}(\mathbb{S}^{n-1})\subsetneqq L(\log ^{+}L)^{\frac{1}{2
}}(\mathbb{S}^{n-1}).
\end{equation*}
We shall also prove the following  two theorems, which give a complete
characterization of the relationship between block spaces and the relevant
logarithmic Orlicz spaces, as well as among the block spaces themselves:

\begin{maintheorem}\label{thm:C} 
\[
\left(
\bigcap_{\beta>0}L(\log L)^\beta(\mathbb S^{n-1})
\right)
\setminus
\left(
\bigcup_{\substack{v>-1\\1<q\leq\infty}}
B_q^{0,v}(\mathbb S^{n-1})
\right)
\neq\varnothing.
\]
\end{maintheorem}

It should be observed that Theorem \ref{thm:C} implies that $\bigcup\limits_{1<q\leq
\infty }B_{q}^{0,\nu }(\mathbb{S}^{n-1})$ is a proper subspace of $
L\left( \log ^{+}L\right) ^{1+\nu }(
\mathbb{S}^{n-1})$. Thus, even after taking the union over all $q>1,$ the
resulting block space remains strictly smaller than the corresponding
related logarithmic Orlicz spaces.

\begin{maintheorem}\label{thm:D} 
	Let \(\nu>-1\) and \(1<p<q\leq\infty\). Then
\(
B_q^{0,\nu}(\mathbb S^{n-1})
\subsetneqq
B_p^{0,\nu}(\mathbb S^{n-1}).
\)
\end{maintheorem}

We end this section by pointing out that, after the completion of this
paper, it came to our attention that the inclusion stated in Theorem B above
was established in the recent preprint of Chen, Ji, Wang and Wu \cite
{Preprint}. The present work was developed independently, and the proof
given here follows a different route. In particular, our argument is direct
and elementary, and is distinct from the approach developed in \cite
{Preprint}.

\section{Proof of Theorem \ref{thm:B}}

Let $\varphi _{\beta }(t)=t\log ^{\beta }(e+t),\beta >0$. Then $\varphi
_{\beta }$ is convex and increasing on $[0,\infty )$ with $\varphi _{\beta
}(0)=0$. We shall need the following lemma:

\begin{lemma}
\label{lem:lemma 0.2}

(i) $\varphi _{\beta }(t)\leq (2^{\beta }+\frac{1}{\alpha ^{\beta }}
)2^{\beta }\left( t+t^{\alpha \beta +1}\right) ,$ $t\geq 0$ and $\alpha >0$.

(ii) $\varphi _{\beta }(\lambda t)\leq (\lambda +\lambda \left( \log
^{+}\lambda \right) ^{\beta })\,\varphi _{\beta }(t),$ $t\geq 0$ and $
\lambda \geq 0$.
\end{lemma}

\begin{proof}
The verification of (i) is straightforward. If $0\leq t\leq 1$ and $\alpha
>0$, then $\varphi (t)\leq 2^{\beta }t\leq (2^{\beta }+\frac{1}{\alpha
^{\beta }})2^{\beta }(t+t^{\alpha \beta }).$ On the other hand, if $t\geq 1$
and $\alpha >0$, then 
\begin{eqnarray}
\log ^{\beta }(e+t) &\leq &\log ^{\beta }((e+1)t)=\left( \log (e+1)+\log
t\right) ^{\beta }  \notag \\
&\leq &\left( 2+\log t\right) ^{\beta }\leq 2^{2\beta }t+\frac{2^{\beta
}t^{\alpha \beta }}{\alpha ^{\beta }}\leq (2^{\beta }+\frac{1}{\alpha
^{\beta }})2^{\beta }t^{\alpha \beta }.  \label{ineq-1}
\end{eqnarray}
Here, we used the inequality $\log (t)\leq \frac{t^{\alpha }}{\alpha }$ for $
t\geq 1$ and $\ \alpha >0$. Thus, (i) is verified. Next, to verify (ii), we
first note that by convexity of $\varphi $ we have 
\begin{equation}
\varphi _{\beta }(\lambda t)\leq \lambda \varphi _{\beta }(t)  \label{e-1}
\end{equation}
for $0\leq \lambda <1$. Secondly, if $\lambda \geq 1$, then 
\begin{eqnarray}
\varphi _{\beta }(\lambda t) &=&(\lambda t)\log ^{\beta }(e+\lambda t)\leq
(\lambda t)\log ^{\beta }(\lambda (e+t))  \notag \\
&=&(\lambda t)\,\left( \log (\lambda )+\log (e+t)\right) ^{\beta }  \notag \\
&=&(\lambda t)2^{\beta }\log ^{\beta }(\lambda )+(\lambda t)2^{\beta }\log
^{\beta }(e+t)\leq 2^{\beta }(\lambda \log ^{\beta }\lambda +\lambda
)\,\varphi _{\beta }(t).  \label{e-2}
\end{eqnarray}
Combining the estimates (\ref{e-1}) and (\ref{e-2}), the verification of
(ii) is complete. This completes the proof.\end{proof} 
%%%%%%%%%%%%%%%%%%%%%%%%%%%%%%%%%%%%%%%%%%%

\begin{lemma}
\label{lem:lemma 0.4} Let $a$ be a $q$-block supported on a cap $M$ on $
\mathbb{S}^{n-1}$ with $\sigma (M)=|M|>0$. Then there exists a constant $
E_{q}$ independent of $a$ such that 
\begin{equation*}
\int_{S^{n-1}}|a(y^{\prime })|\log ^{\beta }(e+|a(y^{\prime })|)\,d\sigma
(y^{\prime })\leq C_{q,\beta }\left( 1+\left( \log ^{+}\frac{1}{|M|}\right)
^{\beta }\right) ,\quad 1<q<\infty .
\end{equation*}
where $C_{q,\beta }=(1+2^{2\beta +1}+\frac{2^{\beta +1}}{(\frac{q-1}{\beta }
)^{\beta }}).$
\end{lemma}

\begin{proof}
Firstly, we have 
\begin{align}
& \int\limits_{\mathbb{S}^{n-1}}|a(y^{\prime })|\log ^{\beta
}(e+|a(y^{\prime })|)\,d\sigma (y^{\prime })  \notag \\
& =\int\limits_{M}|a(y^{\prime })|\log ^{\beta }(e+|a(y^{\prime
})|)\,d\sigma (y^{\prime })  \notag \\
& =\int\limits_{M}|M||a(y^{\prime })|\log ^{\beta }\left( e+\frac{
|M||a(y^{\prime })|}{|M|}\right) \,\frac{d\sigma (y^{\prime })}{|M|}.
\label{e3}
\end{align}
Notice that, if $0<|M|<1$, then

\begin{align}
& |M||a(y^{\prime })|\log ^{\beta }\left( e+\frac{|M||a(y^{\prime })|}{|M|}
\right)   \notag \\
& \leq |M|\left\vert a(y^{\prime })\right\vert \left( \log
(e+|M||a(y^{\prime })|)+\log \left( \frac{1}{|M|}\right) \right) ^{\beta } 
\notag \\
& \leq 2^{\beta }|M||a(y^{\prime })|\left( \log ^{\beta }(e+|M||a(y^{\prime
})|)+\left( 1+\left( \log ^{+}\frac{1}{|M|}\right) ^{\beta }\right) \right) .
\label{e4}
\end{align}

On the other hand, if $|M|\geq 1$, we have 
\begin{align}
& |M||a(y^{\prime })|\log ^{\beta }\left( e+\frac{|M||a(y^{\prime })|}{|M|}
\right)   \notag \\
& \leq |M||a(y^{\prime })|\left( \log ^{\beta }(e+|M||a(y^{\prime
})|)\right)   \notag \\
& \leq 2^{\beta }|M||a(y^{\prime })|\left( \log ^{\beta }(e+|M||a(y^{\prime
})|)+1+\left( \log ^{+}\frac{1}{|M|}\right) ^{\beta }\right) .  \label{e-5}
\end{align}

Therefore, by (\ref{e4}) and (\ref{e-5}), we have 
\begin{equation}
|M||a(y^{\prime })|\log ^{\beta }\left( e+\frac{|M||a(y^{\prime })|}{|M|}
\right) \leq 2^{\beta }|M||a(y^{\prime })|\log ^{\beta }(e+|M||a(y^{\prime
})|+2^{\beta }\left( 1+\left( \log ^{+}\frac{1}{|M|}\right) ^{\beta }\right)
|M||a(y^{\prime })|,\quad |M|>0.  \label{eq:equation 0.4}
\end{equation}

Finally, combining (\ref{e3}), (\ref{eq:equation 0.4}), and invoking Lemma 
\ref*{lem:lemma 0.2} with $\alpha \beta =q-1$ we get 
\begin{align}
& \int\limits_{S^{n-1}}|a(y^{\prime })|\log ^{\beta }(e+|a(y^{\prime
})|)\,d\sigma (y^{\prime })  \notag \\
& \leq 2^{\beta }\int\limits_{M}|M||a(y^{\prime })|\log ^{\beta
}(e+|M||a(y^{\prime })|\frac{d\sigma (y^{\prime })}{|M|}+2^{\beta }\left(
1+\left( \log ^{+}\frac{1}{|M|}\right) ^{\beta }\right)
\int\limits_{M}|M||a(y^{\prime })|\,\frac{d\sigma }{|M|}  \notag \\
& \leq (2^{\beta }+\frac{1}{(\frac{q-1}{\beta })^{\beta }})2^{\beta
}\int\limits_{M}\left( |a(y^{\prime })|+|M|^{q-1}|a(y^{\prime })|^{q}\right)
d\sigma (y^{\prime })+\left( 1+\left( \log ^{+}\frac{1}{|M|}\right) ^{\beta
}\right) \int\limits_{M}|a(y^{\prime })|d\sigma (y^{\prime })  \notag \\
& \leq (2^{\beta }+\frac{1}{(\frac{q-1}{\beta })^{\beta }})2^{\beta }\left(
\left\vert M\right\vert ^{\frac{1}{q^{\prime }}}\left\Vert a\right\Vert
_{q}+|M|^{q-1}\left\Vert a\right\Vert _{q}^{q}\right) +\left( 1+\left( \log
^{+}\frac{1}{|M|}\right) ^{\beta }\right) \left\Vert a\right\Vert
_{q}\left\vert M\right\vert ^{\frac{1}{q^{\prime }}}  \notag \\
& \leq (1+2^{2\beta +1}+\frac{2^{\beta +1}}{(\frac{q-1}{\beta })^{\beta }}
)\left( 1+\left( \log ^{+}\frac{1}{|M|}\right) ^{\beta }\right) .
\end{align}
\end{proof} %%%%%%%%%%%%%%%%%%%%%%%%%%%%%%%%%%%%%%%

Now, we are ready to prove Theorem B. 

\begin{proof}
Fix $1<q<\infty $ and $\nu >-1$. Let $\Omega \in B_{q}^{0,\nu }(\mathbb{S}
^{n-1})$ and $\Omega =\sum_{j=1}^{\infty }\mu _{j}a_{j}$ be a $q$-block
decomposition in $L^{1}(\mathbb{S}^{n-1})$. By Lemma \ref{lem:lemma 0.4}, we
have 
\begin{equation*}
\int_{\mathbb{S}^{n-1}}\varphi _{\nu +1}(|a_{j}|)\,d\sigma \leq C_{q,1+\nu
}\left( 1+\left( \log ^{+}\frac{1}{|M_{j}|}\right) ^{1+\nu }\right) \coloneqq
C_{q,1+\nu }b_{j}.
\end{equation*}
By convexity of $\varphi _{\nu +1}$ and since $C_{q,1+\nu }b_{j}\geq 1$, we have 
\begin{equation*}
\int\limits_{S^{n-1}}\varphi _{\nu +1}\left( \frac{|a_{j}|}{C_{q,1+\nu }b_{j}
}\right) d\sigma (y^{\prime })\leq 1.
\end{equation*}
Set 
\begin{equation*}
S_{N}=\sum_{j=1}^{N}\mu _{j}a_{j},\qquad \gamma =\sum_{j=1}^{\infty }|\mu
_{j}|b_{j},\qquad \beta _{j}=\frac{|\mu _{j}|b_{j}}{\gamma }.
\end{equation*}

By Jensen's inequality 
\begin{equation*}
\varphi _{\nu +1}\left( \frac{1}{C_{q,1+\nu }}\cdot \frac{|S_{N}|}{\gamma }
\right) \leq \varphi _{\nu +1}\left( \sum_{j=1}^{N}\frac{1}{C_{q,1+\nu }}
\cdot \frac{|\mu _{j}|\,|a_{j}|}{\gamma }\right) \leq \sum_{j=1}^{N}\beta
_{j}\varphi _{\nu +1}\left( \frac{1}{C_{q,1+\nu }}\cdot \frac{|a_{j}|}{b_{j}}
\right) .
\end{equation*}

Therefore, 
\begin{equation*}
\int\limits_{\mathbb{S}^{n-1}}\varphi _{\nu +1}\left( \frac{1}{C_{q,1+\nu }}
\cdot \frac{|S_{N}|}{\gamma }\right) d\sigma (y^{\prime })\leq
\sum_{j=1}^{N}\beta _{j}\int\limits_{S^{n-1}}\varphi _{\nu +1}\left( \frac{1
}{C_{q,1+\nu }}\cdot \frac{|a_{j}|}{b_{j}}\right) d\sigma (y^{\prime })\leq
1.
\end{equation*}

Since $\sum_{j=1}^{\infty }|\mu _{j}|\Vert a_{j}\Vert _{1}\leq \gamma ,$ the
series $\sum_{j=1}^{\infty }\mu _{j}a_{j}$ converges absolutely in $L^{1}(
\mathbb{S}^{n-1})$. Subsequently, $(S_{N})$ converges to $\Omega $ almost
everywhere. Thus, by Fatou's lemma we have 
\begin{align}
\int\limits_{\mathbb{S}^{n-1}}\varphi _{\nu +1}\left( \frac{1}{C_{q,1+\nu }}
\cdot \frac{|\Omega |}{\gamma }\right) d\sigma &
=\int_{S^{n-1}}\lim_{N\rightarrow \infty }\varphi _{\nu +1}\left( \frac{1}{
C_{q,1+\nu }}\,\frac{|S_{N}|}{\gamma }\right) d\sigma   \notag
\label{eq:equation 0.6} \\
& {\leq }\liminf_{N\rightarrow \infty }\int_{S^{n-1}}\varphi _{\nu +1}\left( 
\frac{1}{C_{q,1+\nu }}\cdot \frac{|S_{N}|}{\gamma }\right) d\sigma \leq 1.
\end{align}
Finally, by equation \eqref{eq:equation 0.6} and (ii) of Lemma \ref
{lem:lemma 0.2}, we have 
\begin{align*}
& \int\limits_{\mathbb{S}^{n-1}}|\Omega (y^{\prime })|\log ^{1+\nu
}(e+|\Omega (y^{\prime })|)\,d\sigma (y^{\prime }) \\
& =\int\limits_{\mathbb{S}^{n-1}}\varphi _{\nu +1}(|\Omega (y^{\prime
})|)\,d\sigma (y^{\prime }) \\
& =\int\limits_{\mathbb{S}^{n-1}}\varphi _{\nu +1}\left( C_{q,1+\nu }\cdot 
\frac{1}{C_{q,1+\nu }}\cdot \gamma \cdot \frac{|\Omega (y^{\prime })|}{
\gamma }\right) \,d\sigma (y^{\prime }) \\
& \leq C_{q,1+\nu }\cdot \gamma \left( 1+\log ^{1+\nu }\left( C_{q,1+\nu
}\cdot \gamma \right) \right) <\infty .
\end{align*}
\end{proof}

\section{Proof of Theorems \ref{thm:C} and \ref{thm:D}}

For simplicity, we shall carry out the argument in this section for $n=2$. The case $n>2$ follows by suitable modifications. We shall  need the following proposition which we prove in Section \ref{sec:section4}.

\begin{proposition}\label{prop:prop 3.1}
	Let
	\[
	m_k=\frac{1}{1000k^2},
	\qquad
	\qquad k=3,\dots.
	\]
	Then there exists a sequence \((E_k)_{k\geq3}\) of pairwise disjoint measurable sets with
	\[
	|E_k|=m_k
	\]
	such that, for every \(\eta>0\), there exists a constant \(C_\eta>0\), independent of \(k\) and \(I\), for which
	\[
	|E_k\cap I|
	\leq
	C_\eta\,|E_k|\,|I|
	\left(1+\log\frac1{|I|}\right)^\eta
	\]
	for every arc \(I\) in \(\mathbb S^1\) and every \(k\geq3\). Moreover, each \(E_k\) is a finite union of pairwise disjoint closed arcs in \(\mathbb S^1\).
\end{proposition}

	Now, we present the proof of Theorem \ref{thm:C}.

    \begin{proof}[\textbf{Proof of Theorem C}]
    Let \((E_k)_{k\geq3}\) be as in Proposition \ref{prop:prop 3.1} and consider the function
	\[
	\Omega(y')
	:=
	\sum_{k=3}^{\infty}\gamma_k\chi_{E_k}(y'),
	\]
	where
	\[
	\gamma_k
	:=
	k\exp\left(-\sqrt{\log k}\right).
	\]
	
	We first prove that
	\[
	\Omega\in
	\bigcap_{\beta>0}L(\log L)^\beta(\mathbb S^1).
	\]
	Fix \(\beta>0\). Since the sets \(E_k\) are pairwise disjoint and \(\gamma_k\leq k\), we have
	\begin{align*}
		\int_{\mathbb S^1}
		|\Omega|
		\left(\log(e+|\Omega|)\right)^\beta
		\,d\sigma
		&=
		\sum_{k=3}^{\infty}
		\gamma_km_k
		\left(\log(e+\gamma_k)\right)^\beta\\
		&=
		\frac1{1000}
		\sum_{k=3}^{\infty}
		\frac{\exp(-\sqrt{\log k})}{k}
		\left(\log(e+\gamma_k)\right)^\beta\\
		&\leq
		\frac{2^\beta}{1000}
		\sum_{k=3}^{\infty}
		\frac{\exp(-\sqrt{\log k})}{k}
		(\log k)^\beta.
	\end{align*}
	The last series is convergent. In fact, by the integral test and the change of variables
	\(
	u=\sqrt{\log x},
	\)
	we have
	\begin{align*}
		\int_3^\infty
		\frac{\exp(-\sqrt{\log x})}{x}
		(\log x)^\beta\,dx
		&=
		2\int_{\sqrt{\log3}}^\infty
		e^{-u}u^{2\beta+1}\,du\\
		&<\infty.
	\end{align*}
	Therefore,
	\[
	\Omega\in L(\log L)^\beta(\mathbb S^1)
	\]
	for every \(\beta>0\).
	
	We next prove that
	\[
	\Omega\notin B_q^{0,v}(\mathbb S^1)
	\]
	for every \(v>-1\) and every \(1<q<\infty\).
	
	Fix \(v>-1\) and \(1<q<\infty\). Suppose, to the contrary, that
	\[
	\Omega\in B_q^{0,v}(\mathbb S^1).
	\]
	Let
	\[
	\Omega=\sum_{j=1}^{\infty}\mu_j a_j
	\]
	be an admissible \(q\)-block decomposition, where \(a_j\) is supported in an arc \(I_j\), and
	\[
	\sum_{j=1}^{\infty}|\mu_j|
	\left\{
	1+\frac{1}{1+v}
	\left(\log^+\frac1{|I_j|}\right)^{1+v}
	\right\}
	<\infty.
	\]
	
	For \(N\geq3\), let
	\[
	S_N
	=
	\sum_{k=3}^{N}
	\gamma_k^{q-1}\chi_{E_k}
	\]
	and
	\[
	A_N
	=
	\sum_{k=3}^{N}\gamma_k^qm_k.
	\]
	
	Set
	\[
	\eta=(1+v)q'>0.
	\]
	By Proposition \ref{prop:prop 3.1}, there exists a constant \(C_\eta>0\) such that
	\[
	|E_k\cap I|
	\leq
	C_\eta m_k|I|
	\left(1+\log\frac1{|I|}\right)^{(1+v)q'}
	\]
	for every \(k\geq3\) and every arc \(I\) in \(\mathbb S^1\).
	
	Since the sets \(E_k\) are pairwise disjoint, we have
	\begin{align*}
		\|S_N\|_{L^{q'}(I_j)}^{q'}
		&=
		\sum_{k=3}^{N}
		\gamma_k^{(q-1)q'}
		|E_k\cap I_j|\\
		&=
		\sum_{k=3}^{N}
		\gamma_k^q
		|E_k\cap I_j|\\
		&\leq
		C_\eta |I_j|
		\left(1+\log\frac1{|I_j|}\right)^{(1+v)q'}
		A_N.
	\end{align*}
	Consequently,
	\[
	\|S_N\|_{L^{q'}(I_j)}
	\leq
	C_\eta^{1/q'}
	|I_j|^{1/q'}
	\left(1+\log\frac1{|I_j|}\right)^{1+v}
	A_N^{1/q'}.
	\]
	
	Therefore, by Hölder's inequality,
	\begin{align}\label{eq:equation 1.1}
		\left|
		\int_{\mathbb S^1}\Omega S_N\,d\sigma
		\right|
		&\leq
		\sum_{j=1}^{\infty}
		|\mu_j|
		\int_{I_j}|a_j||S_N|\,d\sigma
		\nonumber\\
		&\leq
		\sum_{j=1}^{\infty}
		|\mu_j|
		\|a_j\|_q
		\|S_N\|_{L^{q'}(I_j)}
		\nonumber\\
		&\leq
		C_\eta^{1/q'}
		A_N^{1/q'}
		\sum_{j=1}^{\infty}
		|\mu_j|
		\left(1+\log\frac1{|I_j|}\right)^{1+v}.
	\end{align}
	
	Since \(1+v>0\), there exists a constant \(C_v>0\) such that
	\[
	(1+t)^{1+v}
	\leq
	C_v
	\left\{
	1+\frac{1}{1+v}t^{1+v}
	\right\},
	\qquad t\geq0.
	\]
	Thus, equation \eqref{eq:equation 1.1} gives
	\begin{align}\label{eq:equation 1.2}
		\left|
		\int_{\mathbb S^1}\Omega S_N\,d\sigma
		\right|
		&\leq
		C_{q,v}
		A_N^{1/q'}
		\sum_{j=1}^{\infty}|\mu_j|
		\left\{
		1+\frac{1}{1+v}
		\left(\log^+\frac1{|I_j|}\right)^{1+v}
		\right\}.
	\end{align}
	
	On the other hand, since the sets \(E_k\) are pairwise disjoint, we have
	\begin{align}\label{eq:equation 1.3}
		\int_{\mathbb S^1}\Omega S_N\,d\sigma
		&=
		\sum_{k=3}^{N}\gamma_k^qm_k
		\nonumber\\
		&=
		A_N.
	\end{align}
	
	Combining equations \eqref{eq:equation 1.2} and \eqref{eq:equation 1.3}, we obtain
	\[
	A_N^{1/q}
	\leq
	C_{q,v}
	\sum_{j=1}^{\infty}|\mu_j|
	\left\{
	1+\frac{1}{1+v}
	\left(\log^+\frac1{|I_j|}\right)^{1+v}
	\right\}.
	\]
	The right-hand side is finite and independent of \(N\).
	
	However,
	\begin{align*}
		A_N
		&=
		\frac1{1000}
		\sum_{k=3}^{N}
		k^{q-2}
		\exp\left(-q\sqrt{\log k}\right).
	\end{align*}
	For all sufficiently large \(k\), we have
	\[
	\frac{q}{\sqrt{\log k}}
	\leq
	\frac{q-1}{2}.
	\]
	Therefore,
	\[
	\exp\left(-q\sqrt{\log k}\right)
	=
	k^{-q/\sqrt{\log k}}
	\geq
	k^{-(q-1)/2}.
	\]
	It follows that
	\[
	k^{q-2}
	\exp\left(-q\sqrt{\log k}\right)
	\geq
	k^{(q-3)/2}
	\]
	for all sufficiently large \(k\). Since
	\[
	\frac{q-3}{2}>-1
	\]
	for every \(q>1\), we conclude that
	\[
	\sum_{k=3}^{\infty}
	k^{q-2}
	\exp\left(-q\sqrt{\log k}\right)
	=
	\infty.
	\]
	Thus,
	\[
	A_N\longrightarrow\infty,
	\]
	which is a contradiction.
	
	Therefore,
	\[
	\Omega\notin B_q^{0,v}(\mathbb S^1),
	\qquad
	v>-1,\quad 1<q<\infty.
	\]
\end{proof}

Following the same approach as above, we now prove Theorem D.
\begin{proof}[\textbf{Proof of Theorem \ref{thm:D}}]
	Firstly,  we assume that $q<\infty$. Since the inclusion
	\[
	B_q^{0,\nu}(\mathbb S^1)
	\subseteq
	B_p^{0,\nu}(\mathbb S^1)
	\]
	has already been established, it remains only to prove that it is proper.
	
	Let \((E_k)_{k\geq3}\) be the sequence given by Proposition
	\ref{prop:prop 3.1} and 
	\(
	\Omega
	=
	\sum_{k=3}^{\infty}
	k^{1/q}\chi_{E_k}.
	\)
	Then
	\[
	\|\Omega\|_p^p
	=
	\sum_{k=3}^{\infty}k^{p/q}|E_k|
	=
	\frac{1}{1000}
	\sum_{k=3}^{\infty}k^{p/q-2},
	\]
	and therefore
	\(
	\Omega\in 	L^p(\mathbb S^1)
	\subseteq
	B_p^{0,\nu}(\mathbb S^1).
	\)
	We claim that
	\(
	\Omega\notin B_q^{0,\nu}(\mathbb S^1).
	\)
	Suppose, to the contrary, that
	\[
	\Omega
	=
	\sum_{j=1}^{\infty}\mu_j a_j
	\]
	is an admissible \(q\)-block decomposition, where \(a_j\) is supported in
	an arc \(I_j\).
	For \(N\geq3\), set
	\[
	S_N
	=
	\sum_{k=3}^{N}
	k^{\frac{1}{q'}}\chi_{E_k}
	\]
	and
	\[
	A_N
	=
	\sum_{k=3}^{N}k|E_k|
	=
	\frac{1}{1000}
	\sum_{k=3}^{N}\frac1k.
	\]
	Let
	\(
	\eta=(1+\nu)q'.
	\)
	By Proposition \ref{prop:prop 3.1},
	\[
	|E_k\cap I|
	\leq
	C_\eta |E_k|\,|I|
	\left(1+\log\frac1{|I|}\right)^{(1+\nu)q'}
	\]
	for every arc \(I\). Thus, we have
	\begin{align*}
		\|S_N\|_{L^{q'}(I_j)}^{q'}
		&=
		\sum_{k=3}^{N}
		k\,|E_k\cap I_j|\\
		&\leq
		C_\eta |I_j|
		\left(1+\log\frac1{|I_j|}\right)^{(1+\nu)q'}
		A_N.
	\end{align*}
	Therefore,
	\[
	\|S_N\|_{L^{q'}(I_j)}
	\leq
	C_\eta^{1/q'}
	|I_j|^{1/q'}
	\left(1+\log\frac1{|I_j|}\right)^{1+\nu}
	A_N^{1/q'}.
	\]
	
	By Hölder's inequality and the \(q\)-block condition,
	\begin{align*}
		\int_{\mathbb S^1}\Omega S_N\,d\sigma
		&\leq
		\sum_{j=1}^{\infty}
		|\mu_j|
		\|a_j\|_q
		\|S_N\|_{L^{q'}(I_j)}\\
		&\leq
		C_{q,\nu}\,A_N^{1/q'}.
	\end{align*}
	On the other hand, we have
	\[
	\int_{\mathbb S^1}\Omega S_N\,d\sigma
	=
	\sum_{k=3}^{N}k|E_k|
	=
	A_N.
	\]
	Consequently,
	\[
	A_N^{1/q}
	\leq
	C_{q,\nu}
	\]
	for every \(N\).
	
	This is impossible, since
	\[
	A_N
	=
	\frac{1}{1000}
	\sum_{k=3}^{N}\frac1k
	\longrightarrow\infty.
	\]
	Thus,
	\[
	\Omega\notin B_q^{0,\nu}(\mathbb S^1).
	\]
	Therefore,
	\[
	\Omega\in
	B_p^{0,\nu}(\mathbb S^1)
	\setminus
	B_q^{0,\nu}(\mathbb S^1),
	\]
	and the inclusion is proper. Finally, for $q=\infty$, if $B_\infty^{0,v} = B_p^{0,v}$, we would get \[B_\infty^{0,v} \subseteq B_{2p}^{0,v}
	\subsetneqq B_p^{0,v} = B_\infty^{0,v},\] which is a contradiction. This concludes the proof of Theorem \ref{thm:D}.
\end{proof}

\section{Proof of Proposition \ref{prop:prop 3.1}}\label{sec:section4}

For \(e^{is},e^{it}\in\mathbb S^1=\{e^{iu}:0\leq u<2\pi\}\), the angular distance is defined by
\[
d(e^{is},e^{it})
:=
\min\{|s-t|,\,2\pi-|s-t|\}.
\]

For \(a\in[0,2\pi)\) and \(0<\rho<\pi\), the open arc centered at \(e^{ia}\) with angular radius \(\rho\) is defined by
\[
A(a,\rho)
:=
\{e^{it}\in\mathbb S^1:d(e^{it},e^{ia})<\rho\}.
\]
Notice that
\[
|A(a,\rho)|
\coloneqq
\sigma(A(a,\rho))
=
\frac{2\rho}{2\pi}
=
\frac{\rho}{\pi}.
\]

Let \(I\subset\mathbb S^1\) be an arc with
\[
0<s\coloneqq|I|\leq1.
\]
For \(\epsilon>0\), let
\[
I^{(\epsilon)}
:=
\{z\in\mathbb S^1:d(z,I)<\epsilon\}.
\]
Clearly, if
\[
0<\epsilon\leq\pi s,
\]
then
\[
|I^{(\epsilon)}|\leq2s.
\]

The proof of Proposition \ref{prop:prop 3.1} is a consequence of the following lemmas.
\begin{lemma}\label{lem:lemma 2.1}
	For 
	\(
	m_k:=\frac{1}{1000 k^2},
	\)
	there exists a sequence \((E_k)_{k\ge 3}\) of pairwise disjoint sets with \(|E_k|=m_k\) such that each \(E_k\) is a finite union of disjoint closed arcs in the unit circle \(\mathbb{S}^1\).
\end{lemma}
\begin{proof}
	Let
	\[
	F_{k-1} := \bigcup_{j=3}^{k-1} E_j.
	\]
	
	We proceed by induction on $k \ge 3$. Firstly, we construct $E_3$. Since $F_2 = \varnothing$, we choose $E_3$ to be any single closed arc on the circle with  length of $m_3$. Suppose now that for a given $k > 3$, we have constructed the sets $E_3, \dots, E_{k-1}$, each of which  is a finite union of disjoint closed arcs,  pairwise
	disjoint, and
	\[
	|E_j|=m_j
	\qquad\text{for every }3\le j\le k-1.
	\]
	
	The set \(F_{k-1}\) is compact and
	\[
	|F_{k-1}|
	=
	\sum_{j=3}^{k-1}|E_j|
	=
	\sum_{j=3}^{k-1}m_j
	\le
	\sum_{j=3}^{\infty}m_j
	<
	\frac14.
	\]
	Since \(\sigma\) is a regular Borel measure on the compact metric space
	\((\mathbb{S}^1,d)\), there exists a compact set
	\(
	K_k\subset \mathbb{S}^1\setminus F_{k-1}
	\)
	such that
	\(
	|K_k|>\frac12.
	\)
	Let
	\[
	\rho_k^*
	:=
	\operatorname{dist}(K_k,F_{k-1})
	:=
	\min\{d(z,w):z\in K_k,\ w\in F_{k-1}\}.
	\]
	Then \(\rho_k^*>0\). Let \(\delta_k<\min\left\{e^{-m_k^{-k}}\cdot \frac{\pi}{100},\frac{\rho_k^*}{100}\right\}\) and
	\(
	\{e^{it_{k,1}},\ldots,e^{it_{k,N_k}}\}\subset K_k
	\)
	be a maximal \(\delta_k\)-separated subset of \(K_k\).
	
	Since the arcs \(A(t_{k,j},\delta_k)\)  cover \(K_k\) we have
	
	\[
	\frac12
	<
	|K_k|
	\le
	\sum_{j=1}^{N_k}|A(t_{k,j},\delta_k)|
	=
	N_k\frac{\delta_k}{\pi}.
	\]
	Consequently,
	\[
	N_k>\frac{\pi}{2\delta_k}.
	\]
	
	We now define \(E_k\).
	For \(1\le j\le N_k\), let
	\[
	Q_{k,j}
	:=
	\{z\in \mathbb{S}^1:d(z,e^{it_{k,j}})\le \frac{\pi m_k}{N_k}\}.
	\]
	Clearly, we have that 
	\(Q_{k,1},\ldots,Q_{k,N_k}\) are pairwise
	disjoint and  do not intersect  \(F_{k-1}\). 
	
	Define
	\[
	E_k:=\bigcup_{j=1}^{N_k}Q_{k,j}.
	\]
	Then \(E_k\) is a finite union of closed arcs. Since the arcs \(Q_{k,j}\) are
	pairwise disjoint, we have
	\[
	|E_k|
	=
	\sum_{j=1}^{N_k}|Q_{k,j}|
	=
	\sum_{j=1}^{N_k}\frac{m_k}{N_k}
	=
	m_k.
	\]
\end{proof}

\begin{lemma}\label{lem:lemma 2.2}
	Let \(I\) be an arc in \(\mathbb S^1\) and let
	\[
	Q_{k,j},
	\qquad j=1,\ldots,N_k,
	\]
	be as in Lemma \ref{lem:lemma 2.1}. If
	\[
	\mathcal J
	=
	\{j\in\{1,\ldots,N_k\}:Q_{k,j}\cap I\neq\varnothing\}
	\]
	and
	\[
	|I|
	\geq
	\exp\left(-m_k^{-k}\right),
	\]
	then
	\[
	\operatorname{Card}\mathcal J
	\leq
	16N_k|I|.
	\]
\end{lemma}

\begin{proof}
	If \(j\in\mathcal J\), then
	\[
	d(e^{it_{k,j}},I)
	\leq
	\frac{\pi m_k}{N_k}.
	\]
	On the other hand, the open arcs
	\[
	A\left(t_{k,j},\frac{\delta_k}{4}\right),
	\qquad j\in\mathcal J,
	\]
	are pairwise disjoint and
	\[
	A\left(t_{k,j},\frac{\delta_k}{4}\right)
	\subset
	I^{\left(
		\frac{\pi m_k}{N_k}+\frac{\delta_k}{4}
		\right)}.
	\]
	
	Since
	\[
	|I|
	\geq
	\exp\left(-m_k^{-k}\right)
	\]
	and
	\[
	\delta_k
	<
	\frac{\pi}{100}
	\exp\left(-m_k^{-k}\right),
	\]
	we have
	\[
	\frac{\pi m_k}{N_k}
	+\frac{\delta_k}{4}
	<
	2m_k\delta_k+\frac{\delta_k}{4}
	<
	\delta_k
	<
	\pi|I|.
	\]
	Therefore,
	\[
	\left|
	I^{\left(
		\frac{\pi m_k}{N_k}+\frac{\delta_k}{4}
		\right)}
	\right|
	\leq
	2|I|.
	\]
	
	It follows that
	\begin{align*}
		\frac{\delta_k}{4\pi}
		\operatorname{Card}\mathcal J
		&=
		\sum_{j\in\mathcal J}
		\left|
		A\left(t_{k,j},\frac{\delta_k}{4}\right)
		\right|\\
		&\leq
		2|I|.
	\end{align*}
	Hence,
	\[
	\operatorname{Card}\mathcal J
	\leq
	\frac{8\pi}{\delta_k}|I|.
	\]
	Since
	\[
	N_k>\frac{\pi}{2\delta_k},
	\]
	we conclude that
	\[
	\operatorname{Card}\mathcal J
	\leq
	16N_k|I|.
	\]
\end{proof}

\begin{lemma}\label{lem:lemma 2.3}
	Under the same assumptions of Lemmas \ref{lem:lemma 2.1} and \ref{lem:lemma 2.2}, we have
	\[
	|E_k\cap I|
	\leq
	16m_k|I|
	\left(1+\log\frac1{|I|}\right)^{\frac{1}{k}}
	\]
	for every \(k\geq3\) and every arc \(I\) in \(\mathbb S^1\).
\end{lemma}

\begin{proof}
	Firstly, suppose that
	\[
	|I|
	<
	\exp\left(-m_k^{-k}\right).
	\]
	Then
	\[
	\log\frac1{|I|}
	>
	m_k^{-k},
	\]
	and consequently
	\[
	m_k
	\left(1+\log\frac1{|I|}\right)^{\frac{1}{k}}
	\geq1.
	\]
	Therefore,
	\[
	|E_k\cap I|
	\leq
	|I|
	\leq
	m_k|I|
	\left(1+\log\frac1{|I|}\right)^{\frac{1}{k}}.
	\]
	
	Now suppose that
	\[
	|I|
	\geq
	\exp\left(-m_k^{-k}\right).
	\]
	Let
	\[
	\mathcal J
	=
	\{j\in\{1,\ldots,N_k\}:Q_{k,j}\cap I\neq\varnothing\}.
	\]
	By Lemma \ref{lem:lemma 2.2}, we have
	\[
	\operatorname{Card}\mathcal J
	\leq
	16N_k|I|.
	\]
	Thus,
	\begin{align*}
		|E_k\cap I|
		&\leq
		\sum_{j\in\mathcal J}|Q_{k,j}|\\
		&=
		\frac{m_k}{N_k}
		\operatorname{Card}\mathcal J\\
		&\leq
		16m_k|I|\\
		&\leq
		16m_k|I|
		\left(1+\log\frac1{|I|}\right)^{\frac{1}{k}}.
	\end{align*}
\end{proof}

\begin{lemma}\label{lem:lemma 2.4}
	For every \(\eta>0\), there exists a constant \(C_\eta>0\), independent of \(k\) and \(I\), such that
	\[
	|E_k\cap I|
	\leq
	C_\eta m_k|I|
	\left(1+\log\frac1{|I|}\right)^\eta
	\]
	for every \(k\geq3\) and every arc \(I\) in \(\mathbb S^1\).
\end{lemma}

\begin{proof}
	Fix \(\eta>0\). Then
	there exists \(K_\eta\) such that
	\[
	\frac{1}{k}\leq\eta,
	\qquad
	k\geq K_\eta.
	\]
	For \(k\geq K_\eta\), Lemma \ref{lem:lemma 2.3} gives
	\begin{align*}
		|E_k\cap I|
		&\leq
		16m_k|I|
		\left(1+\log\frac1{|I|}\right)^{\frac{1}{k}}\\
		&\leq
		16m_k|I|
		\left(1+\log\frac1{|I|}\right)^\eta.
	\end{align*}
	
	For \(3\leq k<K_\eta\), we use the trivial estimate
	\[
	|E_k\cap I|\leq|I|.
	\]
	Since
	\[
	1+\log\frac1{|I|}\geq1,
	\]
	we obtain
	\[
	|E_k\cap I|
	\leq
	\frac1{m_k}
	m_k|I|
	\left(1+\log\frac1{|I|}\right)^\eta.
	\]
	
	Therefore, the result follows by taking
	\[
	C_\eta
	=
	\max\left\{
	16,\,
	1000(K_\eta-1)^2
	\right\}.
	\]
\end{proof}

Proposition \ref{prop:prop 3.1} follows immediately from Lemmas \ref{lem:lemma 2.1} and \ref{lem:lemma 2.4}.


\begin{thebibliography}{99}
\bibitem{ALQALS-ASIAN} Al-Qassem, H. M., Al-Salman, A. J.: Rough
Marcinkiewicz integrals related to surfaces of revolution. Asian J. Math.
7(2), 219--230 (2003).

\bibitem{jmaap} Al-Qassem, H. M., Al-Salman, A. J.: A note on Marcinkiewicz
integral operators. J. Math. Anal. Appl. 282(2), 698--710 (2003).

\bibitem{ALQREV} Al-Qassem, H., Pan, Y. B.: {$L^p$} estimates for singular
integrals with kernels belonging to certain block spaces. Rev. Mat. Iberoam.
18(3), 701--730 (2002).

\bibitem{hokido} Al-Qassem, H. M., Al-Salman, A. J., Pan, Y. B.: Singular
integrals associated to homogeneous mappings with rough kernels. Hokkaido
Math. J. 33(3), 551--569 (2004).

\bibitem{ALHF} Al-Hasan, A. J., Fan, D.: A singular integral operator
related to block spaces. Hokkaido Math. J. 28(2), 285--299 (1999).

\bibitem{CZ1} Calder\'{o}n, A. P., Zygmund, A.: On the existence of certain
singular integrals. Acta Math. 88, 85--139 (1952).

\bibitem{CZ2} Calder\'{o}n, A. P., Zygmund, A.: On singular integrals. Amer.
J. Math. 78, 289--309 (1956).
\bibitem{Preprint} Chen, Y., Ji, Z., Wang, T., and Wu, H. On the
relationship between block spaces and Orlicz spaces. arXiv preprint
arXiv:2606.16216, 2026.

\bibitem{coifman} Coifman, R. R., Weiss, G.: Extensions of Hardy spaces and
their use in analysis. Bull. Amer. Math. Soc. 83(4), 569--645 (1977).

\bibitem{KS} Keitoku, M., Sato, E.: Block spaces on the unit sphere in {$
\mathbb{R}^{n}$}. Proc. Amer. Math. Soc. 119(2), 453--455 (1993).

\bibitem{LUNG-FAN} Chen, L. K., Fan, D. S.: On singular integrals along
surfaces related to block spaces. Integral Equations Operator Theory 29(3),
261--268 (1997).

\bibitem{LTW} Lu, S. Z., Taibleson, M. H., Weiss, G.: \emph{Spaces Generated
by Blocks}. Beijing Normal University Press, Beijing (1989).

\bibitem{BLOCK RELATION} Ye, X. F., Zhu, X. R.: A note on certain block
spaces on the unit sphere. Acta Math. Sin. (Engl. Ser.) 22(6), 1843--1846
(2006).

\bibitem{MTW} Meyer, Y., Taibleson, M. H., Weiss, G.: Some functional
analytic properties of the spaces {$B_q$} generated by blocks. Indiana Univ.
Math. J. 34(3), 493--515 (1985).

\bibitem{SO} Soria, F.: Characterizations of classes of functions generated
by blocks and associated Hardy spaces. Indiana Univ. Math. J. 34(3),
463--492 (1985).

\bibitem{TW} Taibleson, M. H., Weiss, G.: Certain function spaces associated
with a.e. convergence of Fourier series. In: Conference in Honor of A.
Zygmund, Vol. I, Wadsworth, pp. 95--113 (1983).
\end{thebibliography}
\end{document}